\documentclass{article}
\usepackage[utf8]{inputenc}
\usepackage{geometry}
\usepackage{graphicx}
\graphicspath{ {./images/} }
\usepackage{amsmath} 
\usepackage{amssymb} 
\usepackage{listings}
\usepackage{color}
\usepackage[vcentermath]{youngtab}
\usepackage{tikz}
\usepackage{lscape}
\usepackage{amsthm}
\usepackage{hyperref}
\usepackage{authblk}
\usepackage{float}

\theoremstyle{definition}
\newtheorem{thm}{Theorem}

\newtheorem{lemma}{Lemma}

\newtheorem{prop}{Proposition}

\newtheorem{expl}{Example}
\newtheorem{remark}{Remark}

\title{An SOS counterexample to an inequality of symmetric functions}
\author[1]{Alexander Heaton}
\author[2]{Isabelle Shankar}
\affil[1]{Max Planck Institute for Mathematics in the Sciences, Leipzig}
\affil[2]{University of California, Berkeley}
\date{\today}

\begin{document}

\maketitle

\begin{abstract}
    It is known that differences of symmetric functions corresponding to various bases are nonnegative on the nonnegative orthant exactly when the partitions defining them are comparable in dominance order. The only exception is the case of homogeneous symmetric functions where it is only known that dominance of the partitions implies nonnegativity of the corresponding difference of symmetric functions. It was conjectured by Cuttler, Greene, and Skandera in 2011 that the converse also holds, as in the cases of the monomial, elementary, power-sum, and Schur bases. In this paper we provide a counterexample, showing that homogeneous symmetric functions break the pattern. We use semidefinite programming to find an explicit sums of squares decomposition of the polynomial $H_{44} - H_{521}$ as a sum of 41 squares. This rational certificate of nonnegativity disproves the conjecture, since a polynomial which is a sum of squares cannot be negative, and since the partitions 44 and 521 are incomparable in dominance order.
\end{abstract}


\maketitle

\section{Introduction}

In the article \textit{Inequalities for Symmetric Means} \cite{CGS2011}, by Cuttler, Greene, and Skandera, Muirhead-type inequalities are classified for the different common bases of symmetric functions. We briefly provide some definitions in order to state our main Theorem \ref{thm:counterexample}. First, let $m_\lambda$, $e_\lambda$, $p_\lambda$, $h_\lambda$, and $s_\lambda$ denote the monomial, elementary, power-sum, homogeneous, and Schur polynomials, respectively, associated to a partition $\lambda$. Given a symmetric polynomial $g(x)$, the \emph{term-normalized symmetric polynomial} is 
$$G(x) := \frac{g(x)}{g(\mathbf{1})}$$
where $g(\mathbf{1})$ is the symmetric polynomial evaluated on the all ones vector.
By $G_\lambda \geq G_\mu$, we mean that $G_\lambda(x_1,\ldots,x_n) \geq G_\mu(x_1,\ldots,x_n)$, on the nonnegative orthant.
That is, the inequality holds for any number of variables $n$, but only for $x_i \geq 0$, $i = 1, \ldots, n$.
We denote the term-normalized symmetric polynomials for monomial, elementary, power-sum, homogeneous, and Schur polynomials by $M_\lambda$, $E_\lambda$, $P_\lambda$, $H_\lambda$, and $S_\lambda$, respectively.  
The following theorem is a summary of known results on Muirhead-type inequalities (special cases of which go back to Maclaurin, Muirhead, Newton, and Schur), which are proven in \cite{CGS2011, Hurwitz1891, M1902, S2016}.  In particular, the arithmetic-geometric mean inequality is a special case. 
\begin{thm} \label{thm:inequalitiesCGS} Let $\lambda$ and $\mu$ be partitions such that $|\lambda| = |\mu|$.  Then
$$\begin{array}{c c l}
M_\lambda \leq M_\mu  &\iff  &\mu \succeq \lambda\\
E_\lambda \leq E_\mu  &\iff  &\lambda \succeq \mu\\
P_\lambda \leq P_\mu  &\iff  &\mu \succeq \lambda\\
S_\lambda \leq S_\mu  &\iff  &\mu \succeq \lambda
\end{array}$$
whereas $\mu \succeq \lambda$ implies that $H_\lambda \leq H_\mu$, i.e.,
$$\begin{array}{c c l}
H_\lambda \leq H_\mu  &\Longleftarrow  &\mu \succeq \lambda.\\
\end{array}$$
\end{thm}

The converse for the homogeneous symmetric functions statement was conjectured in \cite{CGS2011}.
The authors also reported that for $d=|\lambda|=|\mu|=1,2,\dots,7$ their conjecture had been proven. For $d=8$ and higher, the question was unresolved. 
\begin{thm} \label{thm:counterexample}
A degree-minimal counterexample exhibiting a polynomial $H_{\mu} - H_{\lambda} \geq 0$ with $\lambda,\mu$ incomparable in dominance order is provided by $H_{44} - H_{521}$.
\end{thm}

We certify the nonnegativity of this polynomial on $\mathbb{R}^3_{\geq 0}$ by writing a related polynomial explicitly as a sum of $41$ squares with rational coefficients. Explicitly, the polynomial we exhibit as a sum of squares of polynomials is  $\big(H_{44} - H_{521}\big) (x_1^2,x_2^2,x_3^2) =$
\begin{multline*}
    \frac{1}{9450} \Big( 17 \, x_{1}^{16} + 9 \, x_{1}^{14} x_{2}^{2} +  \, x_{1}^{12} x_{2}^{4} + 18 \, x_{1}^{10} x_{2}^{6} + 60 \, x_{1}^{8} x_{2}^{8} + 18 \, x_{1}^{6} x_{2}^{10}
    +  \, x_{1}^{4} x_{2}^{12} + 9 \, x_{1}^{2} x_{2}^{14}\\
    + 17 \, x_{2}^{16} + 9 \, x_{1}^{14} x_{3}^{2} - 32 \, x_{1}^{12} x_{2}^{2} x_{3}^{2} - 6 \, x_{1}^{10} x_{2}^{4} x_{3}^{2} + 11 \, x_{1}^{8} x_{2}^{6} x_{3}^{2} 
    + 11 \, x_{1}^{6} x_{2}^{8} x_{3}^{2} \\
    - 48 \, x_{1}^{4} x_{2}^{10} x_{3}^{2} - 32 \, x_{1}^{2} x_{2}^{12} x_{3}^{2} + 9 \, x_{2}^{14} x_{3}^{2} +  \, x_{1}^{12} x_{3}^{4}  - 48 \, x_{1}^{10} x_{2}^{2} x_{3}^{4}
    - 22 \, x_{1}^{8} x_{2}^{4} x_{3}^{4} \\
    - 5 \, x_{1}^{6} x_{2}^{6} x_{3}^{4} - 22 \, x_{1}^{4} x_{2}^{8} x_{3}^{4} - 48 \, x_{1}^{2} x_{2}^{10} x_{3}^{4} +  \, x_{2}^{12} x_{3}^{4}  + 18 \, x_{1}^{10} x_{3}^{6} 
    + 11 \, x_{1}^{8} x_{2}^{2} x_{3}^{6} \\
    - 5 \, x_{1}^{6} x_{2}^{4} x_{3}^{6} - 5 \, x_{1}^{4} x_{2}^{6} x_{3}^{6} + 11 \, x_{1}^{2} x_{2}^{8} x_{3}^{6} + 18 \, x_{2}^{10} x_{3}^{6}  + 60 \, x_{1}^{8} x_{3}^{8} 
    + 11 \, x_{1}^{6} x_{2}^{2} x_{3}^{8} \\
    - 22 \, x_{1}^{4} x_{2}^{4} x_{3}^{8} + 11 \, x_{1}^{2} x_{2}^{6} x_{3}^{8} + 60 \, x_{2}^{8} x_{3}^{8} + 18 \, x_{1}^{6} x_{3}^{10}   - 48 \, x_{1}^{4} x_{2}^{2} x_{3}^{10} - 48 \, x_{1}^{2} x_{2}^{4} x_{3}^{10} \\
    + 18 \, x_{2}^{6} x_{3}^{10} +  \, x_{1}^{4} x_{3}^{12} 
    - 32 \, x_{1}^{2} x_{2}^{2} x_{3}^{12} +  \, x_{2}^{4} x_{3}^{12} + 9 \, x_{1}^{2} x_{3}^{14} + 9 \, x_{2}^{2} x_{3}^{14} + 17 \, x_{3}^{16} \Big).
\end{multline*}

\begin{remark}
Of course, there are other ways to show that a polynomial is nonnegative. However, Theorem \ref{thm:counterexample} states something stronger. Not only is it nonnegative, but also a sum of squares. An ongoing research topic belonging to the general context of Hilbert's 17th Problem is to understand the difference between sums of squares and nonnegativity, see \cite{B2006, B2013, BIJ2015, BR2012, CLR1987, L2017} to name only a few. As an interesting example, the degrees of irreducible components of the boundary of the SOS cone are Gromov-Witten numbers (see \cite{BHORS2012, MP2013}).
\end{remark}

\begin{remark}
We refer the reader to \cite[Ch. 7]{S1999} for details on symmetric functions, but we give some brief comments on the homogeneous symmetric functions, since they are the main focus in this article. Let $h_d$ be the sum of all monomials of degree $d$, with $h_0 = 1$. Then define $h_\lambda = \prod h_{\lambda_i}$ where $\lambda = (\lambda_1,\lambda_2,\dots)$ is a sequence of nonnegative integers, eventually zero, with $\lambda_i \geq \lambda_{i+1}$. The homogeneous and elementary symmetric functions $h_\lambda$ and $e_\lambda$ are in many ways dual. By the fundamental theorem of symmetric functions, the $e_d$ are algebraically independent and generate the algebra of symmetric functions $\Lambda$. Therefore, an algebra endomorphism $\omega:\Lambda \to \Lambda$ is uniquely defined by specifying the images $\omega(e_d)$. Defining $\omega(e_d) = h_d$ gives an involution $\omega^2 = \text{id}$ of $\Lambda$ sending $\omega(e_\lambda) = h_\lambda$ and $\omega(h_\lambda) = e_\lambda$. While the transition matrices between the bases $(m_\lambda)$ and $(e_\lambda)$ are matrices with entries in $\{0,1\}$, those for $(m_\lambda)$ and $(h_\lambda)$ are matrices with entries in $\mathbb{N}$. This also reflects the combinatorial interpretations related to placing balls in boxes without or with repetition. See \cite[Sections 7.5, 7.6]{S1999} for more details.
\end{remark}

We leave the proof of Theorem \ref{thm:counterexample} to the end of the paper in Section \ref{sec:proof}.  Instead we begin in Section \ref{sec:methods} by describing our approach of finding a counterexample and extracting an exact rational SOS certificate.  In Section \ref{sec:posets}, we provide many additional counterexamples that have been certified via numerical means.  While our proof of Theorem \ref{thm:counterexample} is provided by an explicit list of polynomials which square and sum to $\left(H_{44} - H_{521}\right)(x_1^2,x_2^2,x_3^2)$, in fact we found many numerical counterexamples in degrees 8, 9, and 10. We expect that most of these numerical counterexamples, could, with some effort, be converted into provably-correct sums of squares using exact rational arithmetic, as we did for $H_{44} - H_{521}$. Section $\ref{sec:posets}$ displays a poset showing how the dominance partial order on partitions would need to be modified in order to correctly reflect nonnegativity, at least as suggested by our numerical counterexamples. It would be very interesting to see if some modification of dominance order could achieve the correct nonnegativity relationships.

\section{Methods} \label{sec:methods}

In this section we outline the steps taken to find the counterexample given in Theorem \ref{thm:counterexample}.
In order to find such a counterexample, we must search over pairs of partitions $(\mu,\lambda)$ that are incomparable in dominance order and such that $(H_\lambda - H_\mu) \geq 0$, specifically searching over partitions of 8 and polynomials in $\mathbb{R}[x_1,x_2,x_3]$.
To this end, we first recognize that we can certify nonnegativity on the nonnegative orthant by replacing the variables with squares.  That is, to certify that $(H_\lambda - H_\mu)(x_1,x_2,x_3) \geq 0$ for all $x\in \mathbb{R}^3_{\geq 0}$, we instead search for polynomials $(H_\lambda - H_\mu)(x_1^2,x_2^2,x_3^2)$ which are SOS.  See Lemma \ref{lem:nonnegativeoctant} in Section \ref{sec:proof} for a proof.

Now we may utilize the well-known machinery of sums of squares, the study of which has a long history. See, for example, \cite{BPT2013} for more on the theory.  
We will specifically use Proposition \ref{lem:ldl} below, a well-known and extremely important fact in sums of squares.
Let $\mathcal{S}_+^N$ denote the cone of $N\times N$ symmetric positive semidefinite matrices in the space of $N\times N$ symmetric matrices $\mathcal{S}^N$.

\begin{prop} \label{lem:ldl}
Let $h$ be a homogeneous polynomial of degree $2d$ in $n$ variables, $h \in \mathbb{R}[x_1,
\dots,x_n]_{2d}$. Let $m$ be a vector containing all $N = \binom{n+d-1}{d}$ monomials of degree $d$.  Then $h$ is a sum of squares exactly when there exists some $A \in \mathcal{S}_+^N$ such that
$$ h(x_1,\dots,x_n) = m^T A m.$$
\end{prop}

Proposition \ref{lem:ldl} tells us that writing a polynomial as a sum of squares is equivalent to solving a semidefinite program (SDP). That is, we must find a positive semidefinite matrix $A$ that also satisfies the linear constraints defined by equating the coefficients of $h(x_1,\dots,x_n)$ and $m^T A m$. Unfortunately, SDP solvers return numerical solutions, i.e. a matrix with floating point entries. In particular, the matrix will (almost) never exactly reproduce the desired polynomial, a problem when searching for an exact counterexample.  We illustrate this with the following running example.

\begin{expl}
An attempt to reproduce the polynomial $(H_{21} - H_{111})(x_1^2,x_2^2,x_3^2)$, whose nonnegativity  follows from Theorem \ref{thm:inequalitiesCGS}, produced the following polynomial with floating point coefficients
\begin{multline*}
\frac{1}{54} \, x_{1}^{6} + \frac{1}{54} \, x_{2}^{6} + \left(7.888609052210118 \times 10^{-31}\right) \, x_{1}^{3} x_{2}^{2} x_{3} + \left(7.888609052210118 \times 10^{-31}\right) \, x_{1}^{2} x_{2}^{3} x_{3}\\
+ \left(3.944304526105059 \times 10^{-31}\right) \, x_{1}^{3} x_{2} x_{3}^{2} - 0.05555555555555555 \, x_{1}^{2} x_{2}^{2} x_{3}^{2} \\
+ \left(3.944304526105059 \times 10^{-31}\right) \, x_{1} x_{2}^{3} x_{3}^{2} + \left(3.944304526105059 \times 10^{-31}\right) \, x_{1}^{2} x_{2} x_{3}^{3} \\
+ \left(3.944304526105059 \times 10^{-31}\right) \, x_{1} x_{2}^{2} x_{3}^{3} + \frac{1}{54} \, x_{3}^{6}
\end{multline*}
even though the desired polynomial is
$$ \frac{1}{54} \, x_{1}^{6} + \frac{1}{54} \, x_{2}^{6} - \frac{1}{18} \, x_{1}^{2} x_{2}^{2} x_{3}^{2} + \frac{1}{54} \, x_{3}^{6}. $$
In fact, Hurwitz proved nonnegativity of this polynomial via sums of squares \cite{Hurwitz1891}. 
\end{expl}

Therefore, in order to find an \textit{exact} sum of squares certificate of nonnegativity, we must make adjustments. To satisfy Proposition \ref{lem:ldl} we must replace the entries of the matrix itself, while staying in the PSD cone, and continuing to satisfy the requirements of $m^T A m = h$ exactly. One approach to this problem is to use continued fractions to find the best rational approximation (with some user-specified bound $B$ on the size of the denominator) to the entries of the matrix. Geometrically, the SDP may return a matrix on or near the boundary of the PSD cone. By rounding the floating point entries to rational numbers, we risk moving outside the cone, resulting in a matrix which is not positive semidefinite. Therefore, many times this rational rounding procedure will fail.

\begin{remark}
In general, rational certificates for polynomials with rational coefficients do not always exist. This was shown by Scheiderer in \cite{S2012} where he provided explicit minimal examples of degree 4 polynomials in 3 variables. Since the polynomial in Theorem \ref{thm:counterexample} is of degree 16, there is no a priori reason to believe it must have a rational sum of squares representation. 
\end{remark}

Several approaches to this rational rounding problem have been developed. The package \texttt{SOS} has a rational rounding procedure built-in, but for our polynomial their package returned an error stating that the rational rounding had failed. The software \texttt{RealCertify} \cite{MS2018-software}, based on \cite{MS2018}, uses a hybrid numeric-symbolic algorithm for finding rational approximations for polynomials lying in the interior of the SOS cone. In correspondence during the writing of this paper, Mohab Safey El Din reported that \texttt{RealCertify} failed to terminate for our problem. However, in \cite{MS2018} they also describe and compare complexity of several different algorithms, including geometric critical point methods. Safey El Din reported that the geometric critical point methods were successful on our problem, providing a second confirmation of the nonnegativity of our polynomial, although not of its SOS-ness.

\begin{remark}
Another approach would be to search directly for the matrix using exact arithmetic as in \cite{HNS2016} with the package \texttt{SPECTRA} for \texttt{Maple}. However, for a problem of our size, this approach is not promising. Indeed we let \texttt{SPECTRA} run for several days, and it did not terminate. Ours is a feasibility problem, but when optimizing a linear function for a rational SDP, the entries of the optimal solution matrix will be algebraic numbers. In \cite{NRS2006} the algebraic degree of an SDP is introduced. For generic inputs, this degree depends only on the rank $r$ of the solution matrix, the size $n$ of the symmetric matrices, and the dimension $m$ of the affine subspace. In \cite{BR2009} the authors give an exact formula for the algebraic degree. If $n = 45, m = 129, r = 41$ (which corresponds to our problem), their formula yields the following 74 digit number:
$$27986928303724394857777762195272647267703276932951767224059513477726952420.$$
\end{remark}

\addtocounter{expl}{-1}
\begin{expl} \label{expl:matrices} \textbf{(continued)}
Returning to the example above, the output of the SDP for $(H_{21} - H_{111})(x_1^2,x_2^2,x_3^2)$ was the following matrix, for which we print only the first four of ten columns:
\begin{tiny}
$$ \left(\begin{array}{rrrr}
\frac{1}{54} & 0 & 0 & -0.009259228275647818 \\
0 & 0.018518456551295637 & 1.4472934340259067 \times 10^{-17} & -3.1517672055168234 \times 10^{-14} \\
0 & 1.4472934340259067 \times 10^{-17} & 0.018518456551295637 & 3.1499353284307314 \times 10^{-14} \\
-0.009259228275647818 & -3.1517672055168234 \times 10^{-14} & 3.1499353284307314 \times 10^{-14} & 0.018518456551295637 \\
-1.4472934340259067 \times 10^{-17} & 7.150822547960123 \times 10^{-18} & 7.150822547960123 \times 10^{-18} & 7.150822547960123 \times 10^{-18} \\
-0.009259228275647818 & 3.1499353284307314 \times 10^{-14} & -3.1517672055168234 \times 10^{-14} & -0.009259259448737624 \\
3.1517672055168234 \times 10^{-14} & -0.009259228275647818 & -3.1506504106855275 \times 10^{-14} & 0 \\
-3.1506504106855275 \times 10^{-14} & 3.1499353284307314 \times 10^{-14} & -0.009259259448737624 & 1.4472934340259067 \times 10^{-17} \\
-3.1506504106855275 \times 10^{-14} & -0.009259259448737624 & 3.1499353284307314 \times 10^{-14} & 3.1499353284307314 \times 10^{-14} \\
3.1517672055168234 \times 10^{-14} & -3.1506504106855275 \times 10^{-14} & -0.009259228275647818 & -3.1506504106855275 \times 10^{-14}
\end{array} \cdots \right)$$
\end{tiny}
Using continued fractions with denominator bound $B=150$ we obtain the following (preferable) matrix:
$$ A = \frac{1}{108}\left(\begin{array}{rrrrrrrrrr}
2 & 0 & 0 & -1 & 0 & -1 & 0 & 0 & 0 & 0 \\
0 & 2 & 0 & 0 & 0 & 0 & -1 & 0 & -1 & 0 \\
0 & 0 & 2 & 0 & 0 & 0 & 0 & -1 & 0 & -1 \\
-1 & 0 & 0 & 2 & 0 & -1 & 0 & 0 & 0 & 0 \\
0 & 0 & 0 & 0 & 0 & 0 & 0 & 0 & 0 & 0 \\
-1 & 0 & 0 & -1 & 0 & 2 & 0 & 0 & 0 & 0 \\
0 & -1 & 0 & 0 & 0 & 0 & 2 & 0 & -1 & 0 \\
0 & 0 & -1 & 0 & 0 & 0 & 0 & 2 & 0 & -1 \\
0 & -1 & 0 & 0 & 0 & 0 & -1 & 0 & 2 & 0 \\
0 & 0 & -1 & 0 & 0 & 0 & 0 & -1 & 0 & 2
\end{array}\right) $$ 
With $m^T = \left(x_{1}^{3},\,x_{1}^{2} x_{2},\,x_{1}^{2} x_{3},\,x_{1} x_{2}^{2},\,x_{1} x_{2} x_{3},\,x_{1} x_{3}^{2},\,x_{2}^{3},\,x_{2}^{2} x_{3},\,x_{2} x_{3}^{2},\,x_{3}^{3}\right)$, we can calculate $ m^T A m $, obtaining
$$ \frac{1}{54} \Big( x_{1}^{6} +  \, x_{2}^{6} - 3 \, x_{1}^{2} x_{2}^{2} x_{3}^{2} +  \, x_{3}^{6} \Big) $$
which is exactly $(H_{21} - H_{111})(x_1^2,x_2^2,x_3^2)$, as desired. Since $A$ is positive semidefinite, by carrying out $LDL^T$ factorization we obtain the following sum of squares representation of $(H_{21} - H_{111})(x_1^2,x_2^2,x_3^2)$:
\begin{multline*}
\frac{1}{216} \, {\left(2 \, x_{1}^{3} - x_{1} x_{2}^{2} - x_{1} x_{3}^{2}\right)}^{2} + \frac{1}{216} \, {\left(2 \, x_{1}^{2} x_{2} - x_{2}^{3} - x_{2} x_{3}^{2}\right)}^{2} + \frac{1}{72} \, {\left(x_{1} x_{2}^{2} - x_{1} x_{3}^{2}\right)}^{2}\\
+ \frac{1}{72} \, {\left(x_{2}^{3} - x_{2} x_{3}^{2}\right)}^{2} + \frac{1}{216} \, {\left(2 \, x_{1}^{2} x_{3} - x_{2}^{2} x_{3} - x_{3}^{3}\right)}^{2} + \frac{1}{72} \, {\left(x_{2}^{2} x_{3} - x_{3}^{3}\right)}^{2}.
\end{multline*}

Of course, the expression of a polynomial as a sum of squares is not unique.  
For example, this same polynomial appears as a sum of squares of binomials in \cite{Hurwitz1891} and also as a sum of squares of binomials and one trinomial in \cite{ReznickQuantitative1987Hurwitz}.
\begin{align*}
    x_{1}^{6} +  \, x_{2}^{6} +  \, x_{3}^{6} - 3 \, x_{1}^{2} x_{2}^{2} x_{3}^{2}    &= \,(x_1^3 - x_1x_2^2)^2   + (x_3^3 - x_2^2x_3)^2 \\ 
    &\qquad + \frac{1}{2}(x_1^2x_2-x_2^3)^2  + \frac{1}{2}(x_2x_3^2 - x_2^3)^2 + \frac{3}{2}(x_1^2x_2 - x_2x_3^2)^2 \\
    & = (x_1^2x_2 - x_2^3)^2 + (x_1^2x_3 - x_3^3)^2 + \frac{7}{4}(x_1x_2^2-x_1x_3^2)^2  + \frac{1}{4}(x_1x_2^2 + x_1x_3^2 - 2x_1^3)^2
\end{align*} 
\end{expl}

As noted in the discussion above, for $H_{44} - H_{521}$, existing tools do not return a numerical matrix which can be successfully rounded.  Our solution to this problem depended crucially on two things, using the real zeros of the polynomial, and using symmetry. In particular, we used the symmetry reduction techniques developed by Gatermann and Parrilo in \cite{GP2004} which we briefly describe now.

To start, we wish to take advantage of the fact that $H_{44} - H_{521}$ is a symmetric polynomial. That is to say, it is invariant under the action of the symmetric group.  There is a great deal of literature on the subject of symmetric polynomials and sums of squares, including \cite{BR2012, CL1976,L2017, GP2004, GKR2017}, to name only a few.
In particular, we specialize Theorem 3.3 of \cite{GP2004} to our setting as follows:
 
\begin{thm} \label{thm:symmetry}
Given an orthogonal linear representation of the symmetric group $S_n$, $\sigma: S_n \to Aut(\mathcal{S}^N)$, consider a semidefinite program whose objective and feasible matrices are invariant under the group action. Then the optimal value of the SDP is equal to the optimal value of the same SDP restricted to its fixed point subspace, $\{X \in \mathcal{S}^N: X = \sigma(g)X, \,\, \forall g \in S_n\}$.
\end{thm}

In our case, the $S_3$ action on the space of polynomials in three variables induces an action on the symmetric matrix of our SDP, where group elements act on the symmetric matrix by conjugation.  
More specifically, for each $g \in S_3$, let $\rho(g)$ be the associated matrix that permutes the monomials of degree 8 in 3 variables. 
Then the induced action sends a symmetric $45\times 45$ matrix $X \to \rho(g)^TX\rho(g)$.  Note that $\rho(g)$ is an orthogonal matrix.  Then the fixed-point subspace for our particular SDP is
$$\mathcal{F} = \{X : X\rho(g) = \rho(g)X, \,\, \forall g \in S_3\}.$$
Theorem \ref{thm:symmetry} guarantees that if a solution exists to our SDP, a solution also exists if we restrict to this fixed-point subspace.  Thus we force our matrix $A$ to commute with the elements of our group, obtaining better constraints on our semidefinite program.  
Indeed, we further simplify our SDP with additional linear constrains via Lemma \ref{lem:zeros}, but leave the details in Section \ref{sec:proof}.
The resulting simplified SDP returns a positive semidefinite matrix to which we can successfully apply the rational rounding described above.  Upon factoring the matrix using exact, rational arithmetic we obtain an explicit sum of squares representation for $H_{44} - H_{521}$ evaluated at $(x_1^2,x_2^2,x_3^2)$.

\section{Poset of SOS Certifications}\label{sec:posets}

Theorem \ref{thm:counterexample} offers a counterexample in the pair of partitions (521, 44).
We provide an exact, rational certificate of nonnegativity by applying rational rounding to the numerical solution returned by the SDP solver, and then factoring the resulting matrix to obtain a provably-correct expression of $(H_{44} - H_{521})(x_1^2,x_2^2,x_3^2)$ as a sum of squares.  However, our search over other pairs of partitions returned many other counterexamples, in degrees 8, 9, and 10. We left these other counterexamples in floating point, though we found an exact, rational SOS certificate in the case of $H_{44} - H_{521}$.
Below we provide a partially ordered set of all differences of term-normalized homogeneous symmetric polynomials of degree $8, 9, 10$ in three variables that are SOS. That is, for each arrow going from $\lambda$ to $\mu$, $\big(H_\mu - H_\lambda\big)(x_1^2,x_2^2,x_3^2)$ is an SOS polynomial, as certified numerically.

\begin{figure}[H]
	\centering
		\includegraphics[width=0.7\textwidth]{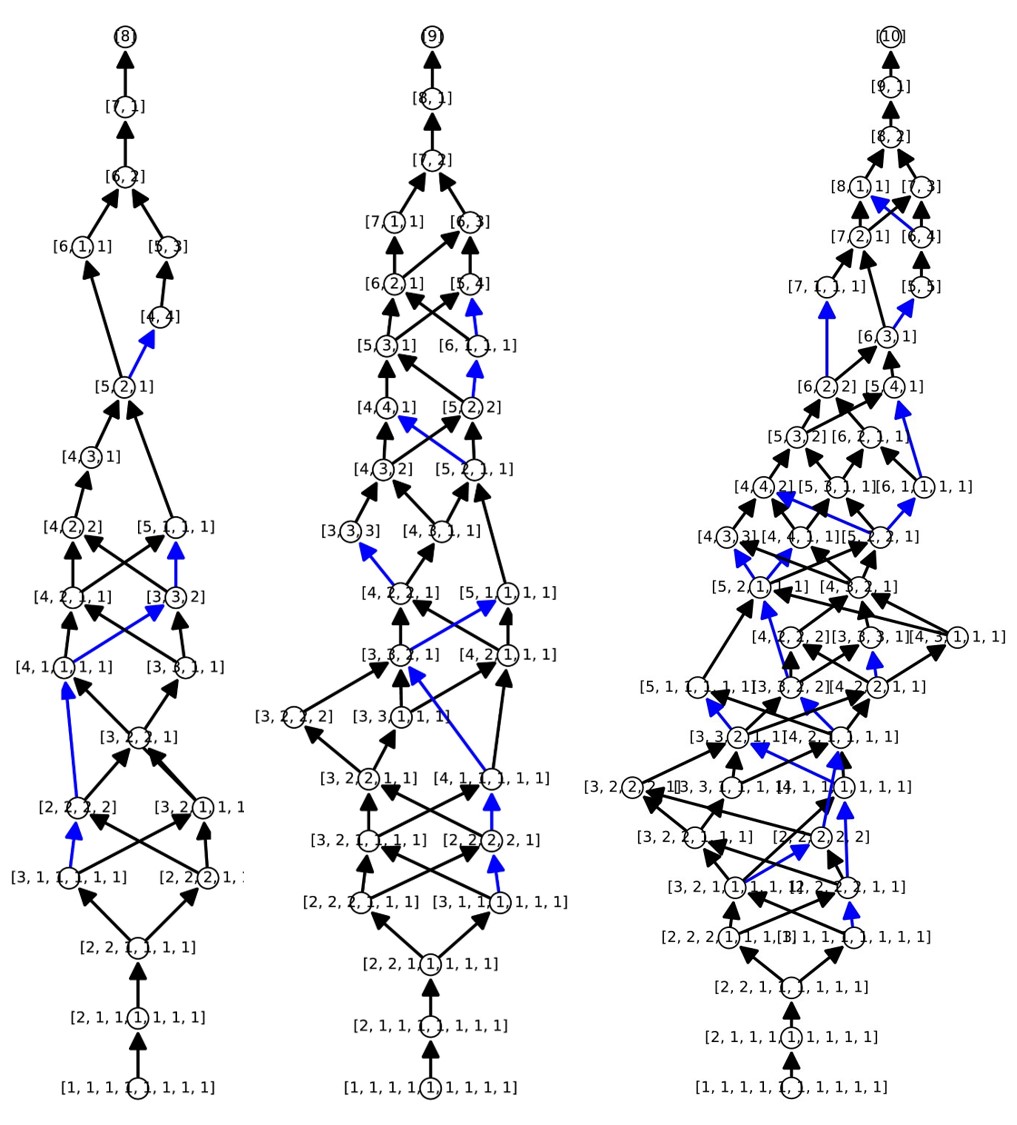}
	\label{fig:HSOS_Poset}
\end{figure}

The black arrows coincide with the dominance order, and therefore certify that these polynomials are not only nonnegative as stated in Theorem \ref{thm:inequalitiesCGS}, but in fact SOS.
Additionally, the blue arrows (numerically) certify SOS-ness for incomparable pairs of partitions, i.e. each blue arrow is a counterexample to the conjecture.

\begin{remark}
 A natural question arises: Is there another partial order on the set of partitions which matches the poset of nonnegativity relations amongst the $H_\lambda$ above?
 A first idea would be to modify the dominance (also called majorization) order slightly, to incorporate the correct relationships among the $H_\lambda$. Since dominance order is related to the cumulative sums produced by the vectors $$(1,0,0,0,\dots), \, (1,1,0,0,\dots), \, (1,1,1,0,\dots), \, \dots$$
 perhaps a modification of the components of these vectors might be a step in the correct direction. We thank the anonymous reviewer for suggesting this possibility. Such a modification might also be informed by the relationship of the homogeneous symmetric functions to the other usual bases. 
\end{remark}

\begin{remark}
If we fix the degree and let the number of variables $n$ go to infinity, the homogeneous symmetric functions behave like the elementary symmetric functions.  This is because the number of square free terms will dominate the number of monomials with squares for very large $n$.  Thus, as $n$ increases, we will likely see fewer counterexamples.  An interesting question is whether the conjecture is true asymptotically i.e. does $H_\lambda \leq H_\mu$ imply $\mu \succeq \lambda$ for large enough $n$?
\end{remark}

\section{Proof of Theorem \ref{thm:counterexample}} \label{sec:proof}

As outlined in Section \ref{sec:methods}, our proof relies on Proposition \ref{lem:ldl} and on the following well-known fact.

\begin{lemma} \label{lem:nonnegativeoctant}
Consider a polynomial $H(x_1,\ldots,x_n)$. Define another polynomial $$h(x_1,\ldots,x_n) = H(x_1^2,\ldots,x_n^2).$$
If $h$ can be written as a sum of squares, then $H$ is nonnegative on the nonnegative orthant.
\end{lemma}

\begin{proof}
Suppose 
$$ h = \sum d_i q_i^2 $$
for positive $d_i > 0$ and polynomials $q_i(x_1,\ldots,x_n)$.  Then $h$ is nonnegative on all of $\mathbb{R}^n$. By way of contradiction, assume there is some point $(a_1,\ldots,a_n) \in \mathbb{R}^n_{\geq 0}$ where $H(a_1,\ldots,a_n) < 0$. This implies that there exist real numbers $\sqrt{a_1},\ldots,\sqrt{a_n}$. But then $h(\sqrt{a_1},\ldots,\sqrt{a_n}) = H(a_1,\ldots,a_n) < 0$, contradicting the nonnegativity of $h$.
\end{proof}

We also use one more lemma to impose additional constrains and help further reduce the size of our SDP, thus ensuring we can apply rational rounding.

\begin{lemma} \label{lem:zeros}
If $x^*$ is a (nonzero) real root of the polynomial $h = m^T A m$, which we write as a sum of squares using the factorization of $A$, then the monomial vector $m$ evaluated at $x^*$ must be in the nullspace of $A$.
\end{lemma}
\begin{proof}
This follows from $0 = h(x^*) = m(x^*)^T A m(x^*)$ and the fact that $A$ is positive semidefinite.
\end{proof}

\begin{proof}[Proof of the main Theorem \ref{thm:counterexample}]
We ran a semidefinite program to find a symmetric positive semidefinite matrix whose factorization could produce a sums of squares representation for $H_{44} - H_{521}$ evaluated at $x_1^2,x_2^2,x_3^2$. By Lemma \ref{lem:nonnegativeoctant} this certifies the non-negativity of $H_{44} - H_{521}$ on the non-negative octant. By Proposition \ref{lem:ldl}, the existence of a matrix which \textit{exactly} reproduces our polynomial is equivalent to the existence of a sums of squares representation. Also by Proposition \ref{lem:ldl}, we impose linear constraints on the entries of our unknown matrix $A$ to require that $m^T A m$ exactly matches the coefficients of our desired polynomial. By Theorem \ref{thm:symmetry}, we impose linear constraints to force our unknown matrix $A$ to commute with the action of the symmetric group on the space of symmetric $45 \times 45$ matrices. Finally, by Lemma \ref{lem:zeros}, we incorporate 4 linearly independent vectors $m(x^*)$, coming from real zeros of our polynomial, to obtain more linear conditions on our SDP. The output is a numerical matrix sufficiently located in the PSD cone such that continued fractions rational approximation yields a matrix with exact, rational entries.  This matrix remains positive semidefinite and produces our desired polynomial via $m^T A m$. We pause to emphasize that without using symmetry, and without using real zeros, the numerical output of the SDP was not accurate enough to recover an exact solution. Only by using symmetry and real zeros were our methods successful. The entries of this $45 \times 45$ matrix include rational numbers with quite large denominators.  Therefore we do not print the matrix here. Rather, we refer the reader to the first author's website \cite{alex} for the explicit matrix.

To give the reader a feel for the matrix, the first row is:
\begin{multline*}
    \bigg( \frac{17}{9450},\,0,\,0,\,-\frac{4}{1433},\,0,\,-\frac{4}{1433},\,0,\,0,\,0,\,0,\,\frac{5}{6699},\,0,\,\frac{1}{484},\,0,\,\frac{5}{6699},\,0,\,0,\,0,\,0,\,0,\,0,\,\frac{1}{5516}, \\ \,0,\,\frac{4}{6655},\,0,\,\frac{4}{6655},\,0,\,\frac{1}{5516},\,0,\,0,\,0,\,0,\,0,\,0,\,0,\,0,\,\frac{157747519610069845105323375343}{7800425434777364748948750531770400}, \\ \,0,\,-\frac{1}{2243},\,0,\,-\frac{490859542561433043273727488474533399}{1004640661046807224753241364163337033424}, \,0, \,-\frac{1}{2243},\,0, \\ \,\frac{157747519610069845105323375343}{7800425434777364748948750531770400} \bigg)
\end{multline*}

 In order to produce an explicit sum of squares representation for $(H_{44} - H_{521})(x_1^2,x_2^2,x_3^2)$ it remains to factor this matrix. We factored the matrix using exact, rational arithmetic, obtaining from this factorization an explicit sum of squares representation. The first polynomial which is to be squared and linearly combined with other squares is displayed below. The others can be found in a \texttt{.txt} file at \cite{alex}, along with the required (positive) coefficients. For those averse to squaring and summing by hand, we also provide open-source computer code that squares and sums them, verifying Theorem \ref{thm:counterexample}. We have also made this code available at the ``mathrepo'' hosted by Max Planck Institute for Mathematics in the Sciences at \href{https://mathrepo.mis.mpg.de/soscounterexample/index.html}{https://mathrepo.mis.mpg.de/soscounterexample/index.html}.


\begin{multline*}
    \frac{4605419240763602856075916837234045570536179818486337018319281394377295}{47993109370744358093263776366419533489066684222682143432782500766944216} x_1^7 x_2\\ - \frac{1302314037803046313255311795382548501141402236228879297620871691257074336598417}{7369150355385188173743144552068464620468501008189426176987169325021222725489728} x_1^5 x_2^3\\ - \frac{1302314037803046313255311795382548501141402236228879297620871691257074336598417}{7369150355385188173743144552068464620468501008189426176987169325021222725489728} x_1^3 x_2^5\\ + \frac{4605419240763602856075916837234045570536179818486337018319281394377295}{47993109370744358093263776366419533489066684222682143432782500766944216} x_1 x_2^7\\ - \frac{50984253124688929255958546913886462493468582355009036545135651731904625}{191972437482977432373055105465678133956266736890728573731130003067776864} x_1^5 x_2 x_3^2\\ + x_1^3 x_2^3 x_3^2\\ - \frac{50984253124688929255958546913886462493468582355009036545135651731904625}{191972437482977432373055105465678133956266736890728573731130003067776864} x_1 x_2^5 x_3^2\\ - \frac{14501649058686817280526502723849668219220268526201662418929955838792199757943}{74143210859800577976256596722743277540321426588463408289289567524830644860352} x_1^3 x_2 x_3^4\\ - \frac{14501649058686817280526502723849668219220268526201662418929955838792199757943}{74143210859800577976256596722743277540321426588463408289289567524830644860352} x_1 x_2^3 x_3^4\\ + \frac{5437035811814876899195863129168499484298662586364259319969151352416367123384855}{64825494532529077215896724731477274708183844806416358400684005156046068663292451} x_1 x_2 x_3^6
\end{multline*}

\end{proof}

\section*{Acknowledgments}

The authors would like to thank Serkan Hosten, Bernd Sturmfels, Mohab Safey El Din, and Greg Blekherman for many helpful discussions. We also wish to thank the Max Planck Institute (MPI) for Mathematics in the Sciences.  Much of the work was done when Isabelle Shankar visited MPI in Leipzig, Germany.  Finally, we thank the anonymous reviewer for numerous helpful comments that improved the exposition.

\bibliographystyle{plain}
\bibliography{references}

\end{document}